\documentclass[11pt,a4paper]{amsart}
\usepackage[left=10mm,right=10mm]{geometry}
\usepackage[english]{babel}
\usepackage[T1]{fontenc}
\usepackage{color}
\usepackage[utf8]{inputenc}
\usepackage{amsthm}
\usepackage{amsmath}
\usepackage{amssymb}
\title{Combinatorics for general kinetically constrained spin models}
\author{Laure \textsc{Marêché}}
\address{Laure Mar\^ech\'e, LPSM UMR 8001, Universit\'e Paris Diderot, Sorbonne Paris Cit\'e, CNRS, 
75013 Paris, France}
\email{mareche@lpsm.paris}
\thanks{I acknowledge support of the ERC Starting Grant 680275 MALIG}
\theoremstyle{plain}
\newtheorem{theorem}{Theorem}
\newtheorem{proposition}[theorem]{Proposition}
\newtheorem{lemma}[theorem]{Lemma}
\newtheorem{remark}[theorem]{Remark}
\theoremstyle{definition}
\newtheorem{definition}[theorem]{Definition}
\usepackage{graphicx}
\usepackage{dsfont}
\usepackage{hyperref}
\usepackage{xcolor}
\usepackage{tikz}
\usetikzlibrary{decorations}
\usetikzlibrary{decorations.pathreplacing}

\begin{document}
 
\maketitle

\renewcommand{\thesubsection}{\arabic{subsection}}

\begin{center}
\begin{minipage}{0.8\textwidth}
\begin{small}
 \textbf{Abstract.} 
We study the set of possible configurations for a general kinetically constrained model (KCM), 
a non monotone version of the $\mathcal{U}$-bootstrap percolation cellular automata. We solve a combinatorial 
question that is a generalization of a problem addressed by Chung, Diaconis and Graham 
in 2001 for a specific one--dimensional KCM, the East model. Since the general 
models we consider are in any dimension and lack the oriented character of the East dynamics, 
we have to follow a completely different route than the one taken by Chung, Diaconis and Graham. 
Our combinatorial result is used by Marêché, Martinelli and Toninelli 
to complete the proof of a conjecture put forward by Morris.

\medskip

 \textbf{2010 Mathematics subject classification}: Primary 60K35, Secondary 05C75.
 
\textbf{Key words.} Kinetically constrained models, bootstrap percolation.
 \end{small}
 \end{minipage}
\end{center}

\section{Introduction}

In this article, we study a generalization of a combinatorial problem addressed by Chung, Diaconis and 
Graham in \cite{Chung_et_al_2001}, that can be formulated as follows. 
Fix $N \in \mathds{N}$ and consider that any element of $\{-N,\dots,N\}$ (we call them \emph{sites}) 
can be in state 0 or 1. The configuration of states can change with respect to the following rules: 
there cannot be two state changes at the same time, and the state of a site can change only if 
its left neighbor is in state zero. We consider that the sites outside $\{-N,\dots,N\}$ have state 0. 
One of the questions tackled in \cite{Chung_et_al_2001} is: 
if the initial configuration contains only ones in 
$\{-N,\dots,N\}$ and if there can only be $n$ zeroes in $\{-N,\dots,N\}$ at the same time, 
is it possible to place a zero at the origin with these rules ? Chung, Diaconis and Graham proved that 
it is possible if and only if $N \leq 2^n-2$: the bigger $N$ is, the bigger $n$ has to be 
(a non rigorous version of this proof was given previously by Sollich and Evans in \cite{Sollich_et_al1999}).

This problem was motivated by the study of the East model \cite{Jackle_et_al1991}, a stochastic particle 
system defined as follows: each site of $\mathds{Z}$ can be in state 0 or 1, and is updated 
(independently) at rate one by setting it to 0 with probability $q$ and to 1 with probability $1-q$, 
if and only if its left neighbor is at zero. Indeed, the above combinatorial result is one of the 
key ingredients to determine the relevant time scales for the East dynamics \cite{Aldous_et_al2002,Cancrini_et_al2008}. 
The East model belongs to a more general class of interacting particle systems, called kinetically 
constrained models (KCM), that were introduced by physicists to model the liquid/glass transition, 
an important open problem of condensed matter physics (see for example \cite{Ritort_et_al,Berthier_et_al2011} 
for reviews). In order to construct a different KCM, we use the same dynamics as for East, 
but with a different choice of the constraint that has to be satisfied to update a site. 
For example, if one allows a site to change state when its left or its right neighbor is at 0 
(this is the choice corresponding to the so-called Fredrickson-Andersen one spin facilitated model 
 (FA1f)), the behavior is entirely different: for any value of $N$, two zeroes at the same time are always enough 
 to reach the origin. Indeed, we can put the site $-N$ at 0, then put $-N+1$ at 0, then put $-N$ at 1, put $-N+2$ at 0, 
 put $-N+1$ at 1, etc. and we end up reaching the origin, using never more than two zeroes at 
 the same time.

 In this article, we study a generalization of the combinatorial problem of Chung, Diaconis and Graham in higher dimension and 
 with totally general rules. Though our motivation comes from the study of KCM, we stress that the content of this 
 paper is purely deterministic and requires no probabilistic tools. Let us give a precise definition of 
 the class of rules that we address. We set $d \in \mathds{N}^*$, $N \in \mathds{N}$; 
 any site of $\{-N,\dots,N\}^d$ can be in state 0 or 1. 
 There cannot be two state changes at the same time, and the state of a site $s$ can change 
only if there exists $X \in \mathcal{U}$ such that all the sites 
of $s+X$ are in state 0, where $\mathcal{U}=\{X_1,\dots,X_m\}$ 
with $m \in \mathds{N}^*$ and the $X_i$ are finite nonempty subsets of $\mathds{Z}^d \setminus \{0\}$ 
($\mathcal{U}$ is called an \emph{update family} and the $X_i$ are called \emph{update rules}). 
As before, the sites outside $\{-N,\dots,N\}^d$ are considered to be in state 0. 
The rules of the East model correspond to $d=1$ and $\mathcal{U} = \{\{-1\}\}$, and those of the FA1f model 
to $d=1$ and $\mathcal{U} = \{\{-1\},\{1\}\}$. If the initial configuration contains only ones in 
$\{-N,\dots,N\}^d$ and if there can only be $n$ zeroes in $\{-N,\dots,N\}^d$ at the same time, 
is it possible to place a zero at the origin ?

This generalization has become interesting in recent years. Indeed, until a few years ago, 
only specific update families had been studied in KCM. However, there recently was a breakthrough 
in the study of a monotone deterministic counterpart of KCM called bootstrap percolation. For 
any update family $\mathcal{U}$ of $\mathds{Z}^d$, the associated bootstrap 
percolation process is defined as follows: we choose a set $A \subset \mathds{Z}^d$ of sites that we consider 
as intially \emph{infected} (the equivalent of being at zero), we set $A_0=A$, 
and for any $t \in \mathds{N}^*$ we define the set $A_t$ of sites that are infected at time $t$ by 
\[
 A_{t} = A_{t-1} \cup \{s \in \mathds{Z}^d \,|\, \exists X \in \mathcal{U},s+X \subset A_{t-1}\},
\]
which means that at each time $t \in \mathds{N}^*$, the sites that were infected at time $t-1$ 
remain infected at time $t$ and a site $s$ that was not infected at 
time $t-1$ becomes infected at time $t$ if and only if 
there exists $X \in \mathcal{U}$ such that all the sites of $s+X$ are infected at time $t-1$.

The articles \cite{Bollobas_et_al2015} by Bollobás, Smith and Uzzell and \cite{Balister_et_al2016} by 
Balister, Bollobás, Przykucki, and Smith tackled general update families for the 
first time and proved a beautiful universality result. They showed that in $\mathds{Z}^2$, the update 
families can be sorted into three classes (whose definitions are too technical to be given in this 
introduction): subcritical, critical and supercritical, which have different behaviors that we are going to describe. 
The first natural question for a bootstrap percolation model is: if we start the process 
with each site having probability $q$ to be infected, independently of the others, will the process infect the origin 
with probability 1 or is there a positive probability that the origin is never infected even if we 
wait for an infinite time ? Moreover, 
what will be the scale of the first time at which the origin is infected (often called \emph{infection time}) ?
Since bootstrap percolation is monotonic (the more infection we have at the beginning, the more we 
will have at any stage), it can be seen that there exists a critical probability $q_c \in [0,1]$ such that 
if $q < q_c$, the origin is never infected with positive probability and if $q > q_c$ 
the origin is infected with probability 1. \cite{Bollobas_et_al2015,Balister_et_al2016} showed that 
when $\mathcal{U}$ is subcritical, $q_c > 0$, and when $\mathcal{U}$ is critical or supercritical, $q_c=0$. 
Moreover, they proved that when $q$ tends to zero, the infection time scales as $1/q^{\Theta(1)}$ 
 when $\mathcal{U}$ is supercritical and as $\exp(1/q^{\Theta(1)})$ when $\mathcal{U}$ is critical 
 (the latter result was later refined by Bollobás, Duminil-Copin, Morris, and Smith in \cite{Bollobas_et_al2017}).
 
 These results call for the study of KCM with general update families. 
 As in bootstrap percolation, a key quantity for the study of KCM is the first time 
 at which the origin is at zero when the process starts with all sites independently at zero with probability $q$; 
 we denote its mean by $\tau(q)$. Understanding the divergence of $\tau(q)$ 
 when $q$ tends to $q_c$ is particularly relevant, because the critical regime $q \downarrow q_c$ is the most interesting 
 for physicists. An easy result proven by Martinelli and Toninelli in \cite{Martinelli_et_al2016} shows that 
 the infection time in the bootstrap percolation process is a lower bound for $\tau(q)$. 
 However, this lower bound does not always give the actual behavior. 
 Indeed, for the East model, the infection time in the bootstrap 
 percolation scales as $1/q^{\Theta(1)}$ when $q$ tends to 0, but 
 the results of Aldous and Diaconis \cite{Aldous_et_al2002} and Cancrini, Martinelli, Roberto and Toninelli 
 \cite{Cancrini_et_al2008} proved that $\tau(q)$ scales as $\exp(\Theta(\log(1/q)^2))$ 
 when $q$ tends to 0. This lead Morris to formulate conjectures on the scaling of $\tau(q)$
 when $q$ tends to zero for critical and supercritical update families. His conjecture for supercritical update families 
 (conjecture 2.7 of \cite{Morris_preprint}) is that they should be divided in two subclasses: 
 supercritical \emph{unrooted} update families for which $\tau(q)$ has the same scaling as the bootstrap percolation 
 infection time, that is $1/q^{\Theta(1)}$, and supercritical \emph{rooted} update families 
 for which $\tau(q)$ has the same scaling as the East model, $\exp(\Theta(\log(1/q)^2))$. 
 Part of this conjecture was proven: 
 the lower bound for supercritical unrooted update families is given by 
 the bootstrap percolation lower bound of \cite{Martinelli_et_al2016}, and the upper bound 
 for supercritical update families both unrooted and rooted was proven by Martinelli, Morris and Toninelli 
 in \cite{MMT}. However, the lower bound for supercritical rooted update families was still missing. 
 Since a lower bound matching this behavior for the East model was proven in \cite{Cancrini_et_al2010} 
using the combinatorial result of \cite{Chung_et_al_2001}, 
we seeked to generalize this combinatorial result to all supercritical rooted update families.
 
 Indeed, we establish the following result (theorem \ref{key_thm_simple}): 
 if $\mathcal{U}$ is a supercritical rooted update family, if 
 we start with all the sites of $\{-N,\dots,N\}^2$ at state 1 and if we allow only $n$ zeroes at the same time 
 in $\{-N,\dots,N\}^2$, then to be able to put a zero at the origin, it is necessary to have $N = O(n 2^n)$. 
 This result is almost optimal, since \cite{Chung_et_al_2001} proved that for the East model, which is 
 supercritical rooted, $N = 2^n-2$ allows to put a zero at the origin. 
 Actually, our result is valid in an even larger class, namely for all update families that are not supercritical 
 unrooted. Furthermore, in proposition \ref{prop_unrooted} we also explain why our hypothesis is not restrictive, namely 
 why such a result is not valid for supercritical unrooted update families. 
Our result allows us to complete the proof of the conjecture of Morris (with respect to $\tau(q)$), which we do in 
theorem 4.2 of \cite{lbounds_infection_time} with Martinelli and Toninelli. Our result proves even more, 
since it is valid in any dimension for a natural generalization of the 
definition of supercritical unrooted update families. 

Though we generalize the result of \cite{Chung_et_al_2001}, our proof is completely 
different from theirs, as the proof of \cite{Chung_et_al_2001} relies heavily on the orientation 
of the East model and the general update families completely lack orientation. 
Note that even in dimension 1, it is a substantial generalization of the result of \cite{Chung_et_al_2001}, 
 because it applies to a whole class of update families instead of just the East model. 
 
 We begin this article by giving the notations and stating the results, then we detail the 
proof of the result for one-dimensional supercritical rooted update families, 
then we explain how this proof extends to general dimension, and finally we examine the 
supercritical unrooted case.

\section{Notations and result}\label{notations}

We fix $d \in \mathds{N}^*$ and set an update family $\mathcal{U}=\{X_1,\dots,X_m\}$ 
with the $X_i$ finite nonempty subsets of $\mathds{Z}^d\setminus \{0\}$.
Set $\Lambda \subset \mathds{Z}^d$. We consider the configurations of states in $\Lambda$; 
they belong to the set $\{0,1\}^\Lambda$. We denote by 
$1_\Lambda$ the configuration which contains only ones in $\Lambda$, 
and by $0_\Lambda$ (or just 0) the configuration which contains only zeroes in $\Lambda$. 
Furthermore, for all $\eta \in\{0,1\}^\Lambda$, $s \in \Lambda$, we 
use the notation $\eta^s$ for the configuration in $\{0,1\}^\Lambda$ 
that is $\eta$ apart from the state of $s$ that is flipped: 
$(\eta^s)_{s'} = 1 - \eta_s$ if $s' = s$ and $\eta_{s'}$ if $s' \neq s$. 
Moreover, if $\Lambda' \subset \Lambda$ and $\eta \in\{0,1\}^\Lambda$, we denote by 
$\eta_{\Lambda'}$ its restriction to $\Lambda'$.
In addition, if $\Lambda' \subset \mathds{Z}^d$ is disjoint from $\Lambda$, for all 
$\eta \in \{0,1\}^\Lambda$, $\eta' \in \{0,1\}^{\Lambda'}$, we denote by 
$\eta_\Lambda\eta'_{\Lambda'}$ the configuration on 
$\Lambda \cup \Lambda'$ defined by $(\eta_\Lambda\eta'_{\Lambda'})_s = \eta_s$ if $s \in \Lambda$ and 
$(\eta_\Lambda\eta'_{\Lambda'})_s = \eta'_s$ if $s \in \Lambda'$. \\
We say that a move from $\eta \in \{0,1\}^\Lambda$ to $\eta' \in \{0,1\}^\Lambda$ is \emph{legal} 
if $\eta' = \eta$, or if $\eta' = \eta^s$ with $s \in \Lambda$ and there 
exists an update rule $X\in \mathcal{U}$ such that $(\eta_\Lambda 0_{\Lambda^c})_{s + X} = 0_{s + X}$ 
(we may also write $(\eta_\Lambda)_{s + X} = 0$ to simplify the notation); 
that is, a move is legal if it respects the rules described in the introduction, assuming that all sites 
outside of $\Lambda$ are zeroes.

\begin{definition}
If $\eta, \eta' \in \{0,1\}^\Lambda$, a 
\emph{legal path} from $\eta$ to $\eta'$ is a sequence of configurations $(\eta^j)_{0 \leq j \leq m}$ 
such that $m \in \mathds{N}^*$, $\eta^0 = \eta$, $\eta^m = \eta'$, 
and for all $j \in \{0, \dots,m-1\}$, the move 
from $\eta^j$ to $\eta^{j+1}$ is legal. \\
 For any $n \in \mathds{N}$, we say that $(\eta^j)_{0 \leq j \leq m}$ 
 is an \emph{$n$-legal path} if for all $j \in \{0,\dots,m\}$, $\eta^j$ does not 
contain more than $n$ zeroes in $\Lambda$.
\end{definition}

In order to have lighter notation, we use the same notation $\eta^j$ for the $j$-th step of a path and for 
the configuration that is equal to $\eta$ everywhere except at site $j$. In order to avoid confusion, 
$\eta^0$, $\eta^j$, $\eta^{j+1}$ and $\eta^m$ will always denote a step of a path, and 
no other index will be used to describe a step of a path.

For all $n \in \mathds{N}$, we define 
\[ 
 V(n,\Lambda) = \{\eta \in \{0,1\}^\Lambda \,|\, \text{there exists an }n\text{-legal 
path from }1_\Lambda\text{ to }\eta\}.
\]
$V(n,\Lambda)$ is the set of configurations 
of $\{0,1\}^\Lambda$ that are attainable from the configuration containing only ones using at 
most $n$ zeroes. $V(n,\Lambda)$ will be very different depending on the properties of $\mathcal{U}$. 
In this article, we will distinguish between two classes of update families. To define them, we recall the concept 
of stable direction introduced in \cite{Bollobas_et_al2015}:

\begin{definition}
For any $u \in S^{d-1}$, let $\mathds{H}_u = \{x \in \mathds{R}^d \,|\, \langle x,u \rangle < 0 \}$ 
the half-space with boundary orthogonal to $u$. We say that $u$ is a \emph{stable direction} 
for the update family $\mathcal{U}$ when there does not exist $X \in \mathcal{U}$ such that 
$X \subset \mathds{H}_u$. 
\end{definition}

This implies in particular that if we apply the rules in $\mathds{Z}^d$ with the update family $\mathcal{U}$, 
and if we start with only ones in $(\mathds{H}_u)^c$, then no zero can appear in $(\mathds{H}_u)^c$. 
Intuitively, it means that the zeroes cannot move towards direction $u$.
The following definition is an extension to the dimension $d$ of the definition proposed in 
\cite{Morris_preprint}:

\begin{definition}\label{def_classification}
We say that $\mathcal{U}$ is \emph{supercritical unrooted} if there exists a hyperplane of $\mathds{R}^d$ 
that contains all stable directions of $\mathcal{U}$.
\end{definition}

An example of supercritical unrooted update family is the one corresponding to 
the Fredrickson-Andersen one spin facilitated model, 
whose one-dimensional version was presented in the introduction, for which 
$\mathcal{U} = \{\{e_1\},\dots,\{e_d\},\{-e_1\},\dots,\{-e_d\}\}$ where $\{e_1,\dots,e_d\}$ 
is the canonical basis of $\mathds{R}^d$. This update family has no stable directions at all.

We are now ready to state our main result, theorem \ref{key_thm_simple}, which 
is valid for all update families that \emph{are not} supercritical unrooted. This actually
covers many different behaviors; in particular, in two dimensions, 
according to the classification in \cite{Bollobas_et_al2015} 
they include: supercritical update families which have two non opposite stable directions (called supercritical rooted 
in \cite{Morris_preprint}), critical and subcritical update families.

\begin{theorem}\label{key_thm_simple}
 Let $\mathcal{U}$ be any update family that is not supercritical unrooted. There exists a constant 
 $\kappa > 0$ such that for any $n \in \mathds{N}$, every 
 $\eta \in V(n,\{-\lfloor\kappa n 2^n \rfloor,\dots,\lfloor\kappa n 2^n \rfloor\}^d)$ satisfies $\eta_0 = 1$.
\end{theorem}

\begin{remark}
 Our theorem is stated for paths that are $n$-legal when all sites outside of 
 the box $\{-\lfloor\kappa n 2^n \rfloor,\dots,\lfloor\kappa n 2^n \rfloor\}^d$ are considered to be zeroes; 
 it actually remains valid if we consider the $n$-legal paths 
 for any configuration of the states outside of the box. Indeed, if we consider that 
 the sites outside of the box are not all zeroes, the possible moves are more restricted, hence a 
 legal path for such a configuration is also a legal path if there are zeroes outside of the box.
 \end{remark}

 The assumption that $\mathcal{U}$ is not supercritical unrooted in theorem \ref{key_thm_simple} 
 is not restrictive. Indeed, if $\mathcal{U}$ is supercritical unrooted, the behavior 
 is different: 
 
 \begin{proposition}\label{prop_unrooted}
  If $d=1$ or $2$, and if $\mathcal{U}$ is supercritical unrooted, there exists $n \in \mathds{N}^*$ 
  such that for any domain $\Lambda \subset \mathds{Z}^d$ containing the origin, 
  there exists $\eta \in V(n,\Lambda)$ such that $\eta_0 = 0$.
 \end{proposition}
 
 Proposition \ref{prop_unrooted} means that there exists a finite $n$ such that $n$ zeroes are always enough
 to bring a zero to the origin. We expect this result to hold also for $d \geq 3$. A sketch of proof 
 can be found in section \ref{section_unrooted}. 

\section{The one-dimensional case}

Let $\mathcal{U}$ be a one-dimensional, non supercritical unrooted update 
family. Then $\mathcal{U}$ has at least one stable 
direction, which can be 1 or -1. Without loss of generality, we may suppose that -1 is a stable direction.
We denote $r$ the range of the interactions: 
$r = \max \{ \|x-y\|_\infty \,|\, x,y \in X \cup \{0\}, X \in \mathcal{U}\}$. Moreover, for all $n \in \mathds{N}$, 
we write $a_n = r(2^n-1)$, $b_n = rn2^{n-1}$ and $\mathcal{P}_n = \{-a_n,\dots,b_n\}$.

We will prove theorem \ref{key_thm_simple} by induction. For all $n \in \mathds{N}$, we denote 
\[
 \mathcal{H}_n = ``  \text{for any }\Lambda \subset \mathds{Z} \text{ such that }
 \mathcal{P}_n \subset \Lambda, \text{ for any }\eta \in V(n,\Lambda), \eta_0=1\text{''}.
\]
 Proving $\mathcal{H}_n$ for all $n \in \mathds{N}$ will prove the theorem in the one-dimensional case. 
 
 In order to do that, we will need the 
 
 \begin{lemma} \label{key_lemma_dim1}
  Let $n \geq 1$ and suppose $\mathcal{H}_{n-1}$. Then, for all $\Lambda \subset \mathds{Z}$ such that 
  $\mathcal{P}_n \subset \Lambda$, for all $\eta \in 
  V(n,\Lambda) \setminus \{1_{\Lambda}\}$, $\eta$ has at least one zero in 
  $\Lambda \setminus \mathcal{P}_{n-1}$.
 \end{lemma}
 
 This lemma means that if $\mathcal{H}_{n-1}$ holds, in a large enough interval, any configuration attainable 
 using no more than $n$ zeroes must have one of its zeroes outside of $\mathcal{P}_{n-1}$ (except the configuration 
containing only ones, that has no zero at all). This implies that there are 
at most $n-1$ zeroes in $\mathcal{P}_{n-1}$, which will allow us to use $\mathcal{H}_{n-1}$ to prove that 
the origin cannot be reached by zeroes (see figure \ref{figure_thm_dim1}). 

\medskip

We first prove the theorem supposing lemma \ref{key_lemma_dim1} holds; 
we will prove the lemma afterwards. As we announced, 
we will show by induction that $\mathcal{H}_n$ holds for any $n \in \mathds{N}$.
 
 \emph{Case n=0.}
 This is a simple case: if $\Lambda \subset \mathds{Z}$, $\mathcal{P}_0 \subset \Lambda$ 
 and $\eta \in V(0,\Lambda)$, then $\eta$ contains no zero.

\emph{Induction.} Let $n \geq 1$. We suppose $\mathcal{H}_{n-1}$. Let us show $\mathcal{H}_n$.
Let $\Lambda \subset \mathds{Z}$ such that $\mathcal{P}_n \subset \Lambda$, and $\eta \in V(n,\Lambda)$.

\begin{figure}
 \begin{center}
  \begin{tikzpicture}[scale=0.15]
\draw (-26,0)--(35,0);
\draw[thick] (-26,0.5)--(-26,-0.5);
\draw[thick] (-19,0.5)--(-19,-0.5);
\draw[thick] (-8,0.5)--(-8,-0.5);
\draw (0,0.5)--(0,-0.5);
\draw[thick] (10,0.5)--(10,-0.5);
\draw[thick] (31,0.5)--(31,-0.5);
\draw[thick] (35,0.5)--(35,-0.5);
\draw (0,0) node[below]{0};
\draw (-29,0) node{$\Lambda$};
\draw[decorate,decoration={brace}] (-8,0.7)--(10,0.7) node[midway,above]{$\mathcal{P}_{n-1}$};
\draw[decorate,decoration={brace}] (-19,4.6)--(31,4.6) node[midway,above]{$\mathcal{P}_n$};
\draw[dotted] (-19,0)--(-19,4.6);
\draw[dotted] (31,0)--(31,4.6);
 \end{tikzpicture}
  \caption{Proof of the theorem in the one-dimensional case: 
  there must be a zero in $\Lambda \setminus \mathcal{P}_{n-1}$, 
  hence there can be at most $n-1$ zeroes in  
  $\mathcal{P}_{n-1}$. Thus $\mathcal{H}_{n-1}$ implies that there is no zero at 0.}
  \label{figure_thm_dim1}
 \end{center}
\end{figure}
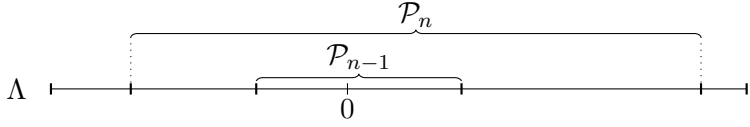

By definition, there exists an $n$-legal path $(\eta^j)_{0 \leq j \leq m}$ from $1_\Lambda$ to $\eta$. 
We will prove that $(\eta_{\mathcal{P}_{n-1}}^j)_{0 \leq j \leq m}$ is an $(n-1)$-legal 
path from $\eta_{\mathcal{P}_{n-1}}^0 = 1_{\mathcal{P}_{n-1}}$ 
to $\eta_{\mathcal{P}_{n-1}}^m = \eta_{\mathcal{P}_{n-1}}$.

Firstly, for all $j \in \{0,\dots,m-1\}$, the move from $\eta_{\mathcal{P}_{n-1}}^j$ to 
$\eta_{\mathcal{P}_{n-1}}^{j+1}$ is legal. Indeed, if $\eta^{j+1} = \eta^j$ or if $\eta^{j+1} = (\eta^j)^z$ 
with $z \in \Lambda \setminus \mathcal{P}_{n-1}$, 
 $\eta_{\mathcal{P}_{n-1}}^{j+1} = \eta_{\mathcal{P}_{n-1}}^j$ and the move from 
 $\eta_{\mathcal{P}_{n-1}}^{j+1}$ to $\eta_{\mathcal{P}_{n-1}}^j$ is legal.
 Furthermore, if $\eta^{j+1} = (\eta^j)^z$ with $z \in \mathcal{P}_{n-1}$, 
 $\eta_{\mathcal{P}_{n-1}}^{j+1} = (\eta_{\mathcal{P}_{n-1}}^j)^z$, and 
 since the move from $\eta^j$ to $\eta^{j+1}$ is legal, 
 there exists $X \in \mathcal{U}$ such that $(\eta_\Lambda^j 0_{\Lambda^c})_{z+X} = 0$, which implies 
 $(\eta_{\mathcal{P}_{n-1}}^j 0_{(\mathcal{P}_{n-1})^c})_{z+X} = 0$, hence 
 the move from $\eta_{\mathcal{P}_{n-1}}^j$ to $\eta_{\mathcal{P}_{n-1}}^{j+1}$ is legal.
 Therefore $(\eta_{\mathcal{P}_{n-1}}^j)_{0 \leq j \leq m}$ is a legal path.

Moreover, for all $j \in \{0,\dots,m\}$, $\eta_{\mathcal{P}_{n-1}}^j$ contains at most $n-1$ zeroes. 
Indeed, if $\eta^j = 1_\Lambda$, then $\eta_{\mathcal{P}_{n-1}}^j$ contains no zero at all. In addition, 
 if $\eta^j \neq 1_\Lambda$, then $\eta^j \in V(n,\Lambda) \setminus \{1_\Lambda\}$, 
 and since we suppose $\mathcal{H}_{n-1}$, we can apply lemma \ref{key_lemma_dim1}, which 
 yields that $\eta^j$ has at least one zero in $\Lambda \setminus \mathcal{P}_{n-1}$, 
 hence $\eta_{\mathcal{P}_{n-1}}^j$ contains at most $n-1$ zeroes.

It follows that $(\eta_{\mathcal{P}_{n-1}}^j)_{0 \leq j \leq m}$ is an $(n-1)$-legal path from 
$1_{\mathcal{P}_{n-1}}$ to $\eta_{\mathcal{P}_{n-1}}$. Thus $\eta_{\mathcal{P}_{n-1}} \in V(n-1,\mathcal{P}_{n-1})$. 
Consequently, by $\mathcal{H}_{n-1}$, $\eta_0=1$, which proves $\mathcal{H}_n$.

This ends the proof of theorem \ref{key_thm_simple} given lemma \ref{key_lemma_dim1}, so we are only left to prove 
lemma \ref{key_lemma_dim1}.

\begin{proof}[Proof of lemma \ref{key_lemma_dim1}]
 
 Let $n \geq 1$ and $\Lambda \subset \mathds{Z}$ be such that $\mathcal{P}_n \subset \Lambda$.
 
 We will consider a configuration $\eta \in \{0,1\}^\Lambda$, 
 different from $1_\Lambda$, containing at most $n$ zeroes, such 
 that all of its zeroes are in $\mathcal{P}_{n-1}$, 
 and we will show that $\eta \not\in V(n,\Lambda)$; this is enough to prove the lemma.
 
 We begin by noticing that if there does not exist an $n$-legal path from $\eta$ to $1_\Lambda$, then 
 $\eta \not\in V(n,\Lambda)$. Indeed, if $\eta \in V(n,\Lambda)$, there exists an $n$-legal path 
 $(\eta^j)_{0 \leq j \leq m}$ from $1_\Lambda$ to $\eta$, and one can check that $(\eta^{m-j})_{0 \leq j \leq m}$ 
 is an $n$-legal path from $\eta$ to $1_\Lambda$. 
 Therefore, to prove that $\eta \not \in V(n,\Lambda)$, it is enough 
 to show that there is no $n$-legal path from $\eta$ to $1_\Lambda$.
 In order to do that, we let $(\eta^j)_{0 \leq j \leq m}$ be an $n$-legal path with $\eta^0=\eta$. We are going to 
 show that $\eta^m$ cannot be $1_\Lambda$. 
 
 To this end, we will denote (see figure \ref{figure_lemme_dim1}): 
 \begin{align*}
  B = & \{-a_n+a_{n-1},\dots,-a_n+a_{n-1}+r-1\} \cup \{b_n-(b_{n-1}+r)+1,\dots,b_n-b_{n-1}\}, \\
  D = & \{-a_n+a_{n-1}+r,\dots,b_n-(b_{n-1}+r)\}, \\
  D_1 = & \{b_n-(b_{n-1}+a_{n-1}+r)+1,\dots,b_n-(b_{n-1}+r)\}, \\
  D_1' = & \{b_n-(b_{n-1}+a_{n-1}+r)+1,\dots,b_n\}
 \end{align*}
and $C = \Lambda\setminus(B \cup D)$ (if $n=1$, $D_1$ will be empty).

\begin{figure}
 \begin{center}
 \begin{tikzpicture}[scale=0.2]
\draw(-26,0)--(35,0);
\foreach \i in {-26,-25,...,35} \draw (\i,0.5)--(\i,-0.5);
\draw (-29,0) node{$\Lambda$};
\draw[very thick] (-19,0.5)--(-19,-0.5);
\draw[very thick] (-11,0.5)--(-11,-0.5);
\draw[very thick] (-8,0.5)--(-8,-0.5);
\draw[ultra thick] (0,0.5)--(0,-0.5);
\draw[very thick] (10,0.5)--(10,-0.5);
\draw[very thick] (18,0.5)--(18,-0.5);
\draw[very thick] (21,0.5)--(21,-0.5);
\draw[very thick] (31,0.5)--(31,-0.5);
\draw [decorate,decoration={brace}] (-26.5,0.7)--(-11.5,0.7) node[midway,above]{$C$};
\draw [decorate,decoration={brace}] (-11.5,0.7)--(-8.5,0.7) node[midway,above]{$B$};
\draw [decorate,decoration={brace}] (-8.5,0.7)--(10.5,0.7) node[midway,above]{$\mathcal{P}_{n-1}$};
\draw [decorate,decoration={brace}] (10.5,0.7)--(18.5,0.7) node[midway,above]{$D_1$};
\draw [decorate,decoration={brace}] (18.5,0.7)--(21.5,0.7) node[midway,above]{$B$};
\draw [decorate,decoration={brace}] (21.5,0.7)--(35.5,0.7) node[midway,above]{$C$};
\draw [decorate,decoration={brace}] (10.5,3.5)--(31.5,3.5) node[midway,above]{$D_1'$};
\draw [dotted] (10.5,3.5)--(10.5,0.7);
\draw [dotted] (31.5,3.5)--(31.5,0.7);
\draw [decorate,decoration={brace}] (-8.5,4.2)--(18.5,4.2) node[midway,above]{$D$};
\draw [dotted] (-8.5,4.2)--(-8.5,0.7);
\draw [dotted] (18.5,4.2)--(18.5,0.7);
\draw [decorate,decoration={brace}] (-19.5,6.8)--(31.5,6.8) node[midway,above]{$\mathcal{P}_n$};
\draw [dotted] (-19.5,6.8)--(-19.5,0.7);
\draw [dotted] (31.5,6.8)--(31.5,0.7);
\draw [<->] (-19,-1)--(-11,-1) node[midway,below]{$a_{n-1}$};
\draw [<->] (-11,-1)--(-8,-1) node[midway,below]{$r$};
\draw [<->] (-8,-1)--(0,-1) node[midway,below]{$a_{n-1}$};
\draw [<->] (0,-1)--(10,-1) node[midway,below]{$b_{n-1}$};
\draw [<->] (10,-1)--(18,-1) node[midway,below]{$a_{n-1}$};
\draw [<->] (18,-1)--(21,-1) node[midway,below]{$r$};
\draw [<->] (21,-1)--(31,-1) node[midway,below]{$b_{n-1}$};
\draw (0,-1) node[below]{0};
\draw [ultra thick] (-19,0.5)--(-19,-0.5);
\draw (-19,-1) node[below]{$-a_n$};
\draw [ultra thick] (31,0.5)--(31,-0.5);
\draw (31,-1) node[below]{$b_n$};
 \end{tikzpicture}
  \caption{The setting of lemma \ref{key_lemma_dim1}.}
  \label{figure_lemme_dim1}
 \end{center}
\end{figure}
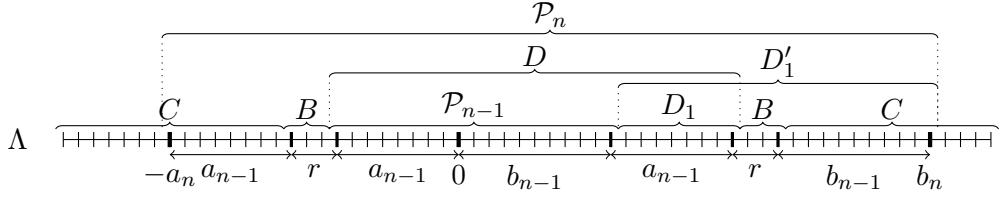

 We notice that 
 \[
  -a_n+a_{n-1}+r = -r(2^n-1)+r(2^{n-1}-1)+r = -r2^{n-1}+r = -r(2^{n-1}-1) = -a_{n-1}
 \]
and 
\[
 b_n-(b_{n-1}+a_{n-1}+r) = rn2^{n-1}-(r(n-1)2^{n-2}+r(2^{n-1}-1)+r)
  = rn2^{n-2}-r2^{n-2} = r(n-1)2^{n-2} = b_{n-1}
\]
hence $\mathcal{P}_{n-1} = \{-a_n+a_{n-1}+r,\dots,b_n-(b_{n-1}+a_{n-1}+r)\} = D \setminus D_1$.

   $B$ will be a ``buffer zone'': we will prove that it remains full of ones and prevents the zeroes of $C$ and 
  $D$ from interacting. \\
  There will always be a zero in $\mathcal{P}_{n-1}$, because 
  the leftmost zero $z$ in $\mathcal{P}_{n-1}$ would need an update rule full of zeroes to disappear. However, 
  there is no zero in $B$ and the thickness of $B$ is larger 
  than the range of the interactions, hence this update rule cannot use zeroes in $B$ 
  or at the left of $B$. Thus it can use only zeroes in $\mathcal{P}_{n-1}$ or at 
  the right of $\mathcal{P}_{n-1}$, but $z$ is the leftmost zero in $\mathcal{P}_{n-1}$. 
  Therefore, the update rule would have to be completely 
  contained in the right of $z$, which is impossible since we assumed that -1 was a stable direction, hence 
  there is no update rule contained in $\mathds{N}^*$. 
  Hence the leftmost zero in $\mathcal{P}_{n-1}$ cannot disappear, thus 
  there will always be a zero in $\mathcal{P}_{n-1}$, which implies $\eta^m\neq1_\Lambda$.
   
  More rigorously, we are going to prove by induction on 
 $j \in \{0,\dots,m\}$ that the property $\mathcal{H}_j'$ holds, where $\mathcal{H}_j'$ consists in:
 \begin{itemize}
  \item[$(P_1^j)$] $\eta_{\mathcal{P}_{n-1}}^j$ contains a zero.
  \item[$(P_2^j)$] $\eta_{B}^j = 1_{B}$.
  \item[$(P_3^j)$] $\eta_C^j 1_{\Lambda \setminus C} \in V(n-1,\Lambda)$.
  \item[$(P_4^j)$] $\eta_{D_1}^j 1_{D_1' \setminus D_1} \in V(n-1,D_1')$.
 \end{itemize}
 The last two properties will be used to show that $B$ remains full of ones.
 
 If we can show $\mathcal{H}_j'$ for all $j \in \{0,\dots,m\}$, in particular $(P_1^m)$ will imply 
 that there is a zero in $\eta_{\mathcal{P}_{n-1}}^m$, thus $\eta^m \neq 1_\Lambda$, 
 which is enough to prove the lemma.  
 
 Let us prove $\mathcal{H}_j'$ for all $j \in \{0,\dots,m\}$.
 
 \emph{Case $j = 0$.} \\
  $(P_1^0)$ is true, because $\eta^0 = \eta \neq 1_\Lambda$, so $\eta$ contains at least a zero,  
  and by assumption all zeroes of $\eta$ are in $\mathcal{P}_{n-1}$.
   $(P_2^0)$ is true because $\eta^0 = \eta$ has no zero in $\Lambda \setminus \mathcal{P}_{n-1}$.
  $(P_3^0)$ is true because $C \subset \Lambda \setminus \mathcal{P}_{n-1}$, thus $\eta_C = 1_C$, hence 
  $\eta_C^0 1 _{\Lambda \setminus C} = 1_\Lambda \in V(n-1,\Lambda)$.
 $(P_4^0)$ is true, because $D_1 \subset \Lambda \setminus \mathcal{P}_{n-1}$, thus 
  $\eta_{D_1}^0 1_{D_1' \setminus D_1} = 1_{D_1'} \in V(n-1,D_1')$.
Consequently, $\mathcal{H}_0'$ holds.
 
 \emph{Induction.} \\
 Let $j$ be in $\{0,\dots,m-1\}$. We suppose that $\mathcal{H}_j'$ holds. Let us show 
 $\mathcal{H}_{j+1}'$.

 We know that the move from $\eta^j$ to $\eta^{j+1}$ is legal. If $\eta^{j+1} = \eta^j$, 
 $\mathcal{H}_{j+1}'$ holds because $\mathcal{H}_j'$ holds. In the following, we deal with the case 
 $\eta^{j+1} = (\eta^j)^z$ where $z \in \Lambda$ and there exists 
 $X \in \mathcal{U}$ with $(\eta^j_\Lambda)_{z+X}=0$. 
 The arguments will depend on the position of $z$. 
 
 \emph{Case $z \in B$.} \\
 We will show that $z \in B$ is impossible: 
 the buffer zone remains preserved at step $j+1$.
  
  By $(P_2^j)$ $\eta_{B}^j = 1_{B}$, hence $z + X \subset C \cup \Lambda^c \cup D$. Moreover, 
  if there existed $x \in (z+X) \cap (C \cup \Lambda^c)$ and $y \in (z+X) \cap D$, then we would get 
  $|x-y| > r$, which is impossible by the definition of $r$. 
  Therefore $z+X \subset C \cup \Lambda^c$ or $z+X \subset D$. We are going to deal with the two cases separately.
  
   We begin with the case $z+X \subset C \cup \Lambda^c$. \\
   We are going to prove that in this case, $(\eta_C^j 1_{\Lambda \setminus C})^z$ would be in $V(n-1,\Lambda)$, 
   which is impossible because it has a zero at $z$ and $z+\mathcal{P}_{n-1} \subset \Lambda$, 
   therefore $\mathcal{H}_{n-1}$ and the invariance by translation of $\mathds{Z}$ 
   yield a contradiction.
  Indeed, the move from $\eta_C^j 1_{\Lambda \setminus C}$ to 
 $(\eta_C^j 1_{\Lambda \setminus C})^z$ would be legal. In addition, 
 $(\eta_C^j 1_{\Lambda \setminus C})^z$ would coincide with $\eta^{j+1}$ on $C \cup B$ by $(P_2^j)$. Moreover,
 $\eta^{j+1}$ contains at most $n$ zeroes, and $\eta^{j+1}_{\mathcal{P}_{n-1}} = 
 \eta^j_{\mathcal{P}_{n-1}}$ would contain at least a zero by 
 $(P_1^j)$, hence $\eta^{j+1}$ contains at most $n-1$ zeroes in 
 $C \cup B$, thus $(\eta_C^j 1_{\Lambda \setminus C})^z$ would contain at most $n-1$ zeroes. 
 Furthermore by $(P_3^j)$, $\eta_C^j 1_{\Lambda \setminus C} \in V(n-1,\Lambda)$. 
 Therefore we could extend an $(n-1)$-legal path from $1_\Lambda$ to $\eta_C^j 1_{\Lambda \setminus C}$ 
 by adding the move from $\eta_C^j 1_{\Lambda \setminus C}$ to 
 $(\eta_C^j 1_{\Lambda \setminus C})^z$ and still have an $(n-1)$-legal path, 
 which would imply $(\eta_C^j 1_{\Lambda \setminus C})^z \in V(n-1,\Lambda)$, which is impossible.
 
 We now deal with the case $z+X \subset D$. \\
 We argue differently depending on the position of $z$. 
 \begin{itemize}
  \item If $z$ is in the left part of $B$, we can use the fact that -1 is a stable direction. Indeed, 
   $z+X$ would be at the right of $z$, hence $X$ would be contained in $\mathds{N}^*$, which yields a contradiction.
  \item If $z$ is in the right part of $B$, we can use an argument similar to the one we used to deal with 
  the case $z+X \subset C \cup \Lambda^c$: $(\eta_{D_1}^j 1_{D_1' \setminus D_1})^z$ would be in $V(n-1,D_1')$, 
   which is impossible because it has a zero at $z$ and $z+\mathcal{P}_{n-1} \subset D_1'$, 
   so by $\mathcal{H}_{n-1}$ there is a contradiction. Indeed, 
   $z+X$ would be contained in $D$ which is disjoint from 
   $D_1' \setminus D_1$, hence the move from $\eta_{D_1}^j 1_{D_1' \setminus D_1}$ to 
 $(\eta_{D_1}^j 1_{D_1' \setminus D_1})^z$ 
 would be legal. Furthermore, $(\eta_{D_1}^j 1_{D_1' \setminus D_1})^z$ would coincide with $\eta^{j+1}$ 
 on $D_1 \cup B$, hence would contain at most $n-1$ zeroes, 
 and by $(P_4^j)$ $\eta_{D_1}^j 1_{D_1' \setminus D_1} \in V(n-1,D_1')$. 
 This would allow us to deduce $(\eta_{D_1}^j 1_{D_1' \setminus D_1})^z \in V(n-1,D_1')$, 
 which is impossible.
 \end{itemize}
We deduce that $z+X \subset D$ is impossible.
  
  Consequently, $z \in B$ is impossible.
  
 \emph{Case $z \in C$.}\\
If $z \in C$, $(P_1^{j+1})$ is true because $\eta_{\mathcal{P}_{n-1}}^{j+1} = 
\eta_{\mathcal{P}_{n-1}}^j$, $(P_2^{j+1})$ is true because 
$\eta_{B}^{j+1} = \eta_{B}^j$, and $(P_4^{j+1})$ is true because $\eta_{D_1}^{j+1} = \eta_{D_1}^j$.
The argument to prove $(P_3^{j+1})$ is almost the same as the one that yielded 
$(\eta_C^j 1_{\Lambda \setminus C})^z\in V(n-1,\Lambda)$ in the case $z \in B$ and $z+X \subset C \cup \Lambda^c$. 
We observe that as $z \in C$, we have $z+X \subset \Lambda^c \cup C \cup B$, and since $(P_2^j)$ implies $\eta_{B}^j = 
1_{B}$, we get $z + X \subset \Lambda^c \cup C$, 
so the move from $\eta_C^j 1_{\Lambda \setminus C}$ to 
$\eta_C^{j+1}1_{\Lambda \setminus C}$ is legal. Furthermore, 
$\eta_C^{j+1}1_{\Lambda \setminus C}$ contains at most $n-1$ zeroes, and by $(P_3^j)$ we have 
$\eta_C^j 1_{\Lambda \setminus C} \in V(n-1,\Lambda)$. This allows us to conclude that 
$\eta_C^{j+1}1_{\Lambda \setminus C} \in V(n-1,\Lambda)$, which is $(P_3^{j+1})$. 
Consequently, $\mathcal{H}_{j+1}'$ holds.

\emph{Case $z \in D$.}\\
If $z \in D$, $(P_2^{j+1})$ is true because $\eta_{B}^{j+1} = 
\eta_{B}^j$, and $(P_3^{j+1})$ is true because $\eta_C^{j+1} = \eta_C^j$.

Let us prove $(P_1^{j+1})$. \\ 
 If $z \in D_1$, then $\eta_{\mathcal{P}_{n-1}}^{j+1} = \eta_{\mathcal{P}_{n-1}}^j$, hence  $(P_1^{j+1})$ is true.
 We now suppose $z \in \mathcal{P}_{n-1}$. We prove $(P_1^{j+1})$ using the fact that -1 
 is a stable direction. Indeed, it implies that $X$ is not contained in $\mathds{N}^*$, hence since 
 $X$ cannot contain 0, it contains an element 
 of $-\mathds{N}^*$, thus there exists $z'\in z+X$ with $z' < z$. In addition, as $z \in \mathcal{P}_{n-1}$ 
 we have $X \subset D \cup B$, and since by $(P_2^j)$ $\eta_{B}^j = 1_{B}$, we get 
 $z+X \subset D$, therefore $z' \in D$. Since $z' < z$, $z' \in \mathcal{P}_{n-1}$, and we have 
 $\eta_{z'}^{j+1}=\eta_{z'}^{j}=0$. Consequently $\eta_{\mathcal{P}_{n-1}}^{j+1}$ 
 contains a zero, hence $(P_1^{j+1})$ is true.

Now let us prove $(P_4^{j+1})$. \\
 If $z \in \mathcal{P}_{n-1}$, then $\eta_{D_1}^{j+1} = \eta_{D_1}^j$, hence $(P_4^{j+1})$ is true. 
In the case $z \in D_1$, we will prove $(P_4^{j+1})$ with the arguments that gave 
$(\eta_{D_1}^j 1_{D_1' \setminus D_1})^z \in V(n-1,D_1')$
in the case $z \in B$ and $z+X \subset D$. Since $z \in D_1$, $z+X \subset D \cup B$, and as $(P_2^j)$ implies 
$\eta_{B}^j = 1_{B}$ we get $z+X \subset D$, thus the move from 
$\eta_{D_1}^j 1_{D'_1 \setminus D_1}$ to $\eta_{D_1}^{j+1} 1_{D'_1 \setminus D_1}$ is legal, 
which allows to prove $\eta_{D_1}^{j+1} 1_{D'_1 \setminus D_1} 
\in V(n-1,D'_1)$. Therefore $(P_4^{j+1})$ is true.

This yields that $\mathcal{H}_{j+1}'$ holds.

To conclude, $\mathcal{H}_{j+1}'$ holds in all cases, which ends the proof of the lemma.
\end{proof}

\section{The general case}

The reasoning to prove theorem \ref{key_thm_simple} in general dimension is the same as in dimension 1. 
However, the geometry is significantly more complicated, which will force us to introduce new notation.

Let $\mathcal{U}$ be a non supercritical unrooted update family. We will need the 

\begin{lemma}\label{lemme_base}
 There exists $u_1,\dots,u_d \in S^{d-1}$ stable directions for $\mathcal{U}$ and 
 a normalized basis $\{v_1,\dots,v_d\}$ of $\mathds{R}^d$ such that for any $i \in \{1,\dots,d\}$, 
$\mathds{H}_{u_i}=\{(x_1,\dots,x_d) \in \mathds{R}^d \,|\, x_i > 0\}$ in this basis.
\end{lemma}

 To construct this basis, one takes $v_i$ orthogonal to all $u_j$ with $j \neq i$.
 A rigorous proof of the construction may be found in the appendix. 
From now on, we will use the coordinates of the basis $\{v_1,\dots,v_d\}$, but when we say a site 
is in $\mathds{Z}^d$, we will mean that its coordinates in the canonical basis are integers. 
For any $i \in \{1,\dots,d\}$, since $u_i$ is a stable direction,  
there is no update rule contained in $\mathds{H}_{u_i}$, hence 
no update rule such that all sites have a positive $i$-th coordinate.

We denote again by $r$ the range of the interactions: 
$r = \max \{ \|x-y\|_\infty \,|\, x,y \in X \cup \{0\}, X \in \mathcal{U}\}$ 
(beware: the range is now defined in our new basis), and for all $n \in \mathds{N}$, 
we set again $a_n = r(2^n-1)$ and $b_n = rn2^{n-1}$. We now have to define $\mathcal{P}_n$ 
as follows (see figure \ref{figure_Pn}): 
\[
 \mathcal{P}_n = \{s \in \mathds{Z}^d \,|\, s=(s_1,\dots,s_d), 
 \forall i \in \{1,\dots,d\}, -a_n \leq s_i \leq b_n \}.
\]

\begin{figure}
 \begin{center}
  \begin{tikzpicture}
\draw (0,0)--(5,0)--(7.5,4.33)--(2.5,4.33)--cycle;
\draw (3,1.73) node{$\times$} node [below left]{$0$};
\draw [->] (3,1.73)--(4,1.73) node[midway,below] {$v_1$} ;
\draw [->] (3,1.73)--(3.5,2.6) node[midway,above left] {$v_2$} ;
\draw [->] (6,1.73)--(5.13,2.23) node[midway,below left] {$u_1$} ;
\draw [->] (4.5,4.33)--(4.5,3.33) node[midway,left] {$u_2$} ;
\draw [dotted] (3,1.73)--(2,0) ;
\draw [<->] (-0.4,0)--(0.6,1.73) node [midway,left] {$a_n$} ;
\draw [<->] (0.6,1.73)--(2.1,4.33) node [midway,left] {$b_n$} ;
\draw [dotted] (3,1.73)--(1,1.73) ;
\draw [<->] (0,-0.3)--(2,-0.3) node [midway,below] {$a_n$} ;
\draw [<->] (2,-0.3)--(5,-0.3) node [midway,below] {$b_n$} ;
\end{tikzpicture}
  \caption{$\mathcal{P}_n$.}
  \label{figure_Pn}
 \end{center}
\end{figure}

We will again prove the theorem by induction: for all $n \in \mathds{N}$, we denote 
\[
 \mathcal{H}_n = 
``\text{for any }\Lambda \subset \mathds{Z}^d\text{ such that }\mathcal{P}_n \subset \Lambda, 
 \text{ for any }\eta \in V(n,\Lambda), \eta_0 = 1 \text{''}.  
\]
 Proving $\mathcal{H}_n$ for all $n \in \mathds{N}$ proves theorem \ref{key_thm_simple}. 
 In order to do that, we need the following equivalent of lemma \ref{key_lemma_dim1}:
 
 \begin{lemma} \label{key_lemma}
 Let $n \geq 1$ and suppose $\mathcal{H}_{n-1}$. Then, for all $\Lambda \subset \mathds{Z}^d$ such that 
  $\mathcal{P}_n \subset \Lambda$, for all $\eta \in 
  V(n,\Lambda) \setminus \{1_{\Lambda}\}$, $\eta$ has at least one zero in 
  $\Lambda \setminus \mathcal{P}_{n-1}$.
 \end{lemma}

 The proof of theorem \ref{key_thm_simple} given lemma \ref{key_lemma} is exactly the same as in 
 the one-dimensional case, therefore it is enough to prove lemma \ref{key_lemma}.
 
\begin{proof}[Proof of lemma \ref{key_lemma}]
 
 Let $n \geq 1$ and $\Lambda \subset \mathds{Z}^d$ such that $\mathcal{P}_n \subset \Lambda$.
 
 As in the one-dimensional case, we consider a configuration $\eta \in \{0,1\}^\Lambda$, 
 different from $1_\Lambda$, containing at most $n$ zeroes, such 
 that all of its zeroes are in $\mathcal{P}_{n-1}$, and we prove that $\eta \not\in V(n,\Lambda)$. 
 As previously, it is enough to let $(\eta^j)_{0 \leq j \leq m}$ be an $n$-legal path with $\eta^0=\eta$, and to 
 show that $\eta^m$ cannot be $1_\Lambda$. 
    
 To this end, we denote for all $i \in \{1,\dots,d\}$ (see figure \ref{figure_lemme}): 
 \begin{align*}
  D & = \{ s \in \mathds{Z}^d \,|\, s=(s_1,\dots,s_d), 
 \forall j \in \{1,\dots,d\}, -a_n+a_{n-1}+r \leq s_j \leq b_n-(b_{n-1}+r)\}, \\
  B & = \{s \in \mathds{Z}^d \,|\, s=(s_1,\dots,s_d),\forall j \in \{1,\dots,d\}
  , -a_n+a_{n-1} \leq s_j \leq b_n-b_{n-1}\} \setminus D, \\
  D_i & = \{s \in D \,|\,s=(s_1,\dots,s_d), s_i > b_n-(b_{n-1}+a_{n-1}+r)\}, \\
  D_i' & = \{s \in \mathcal{P}_n \,|\,s=(s_1,\dots,s_d), s_i > b_n-(b_{n-1}+a_{n-1}+r)\}
 \end{align*}
and $C = \Lambda\setminus(B \cup D)$. We also notice that as in dimension 1, $-a_n+a_{n-1}+r=-a_{n-1}$ and 
$b_n-(b_{n-1}+a_{n-1}+r) = b_{n-1}$, hence  
\[
 \mathcal{P}_{n-1} = \{s \in \mathds{Z}^d \,|\, s=(s_1,\dots,s_d),
 \forall j\in \{1,\dots,d\}, -a_n+a_{n-1}+r \leq s_j \leq b_n-(b_{n-1}+a_{n-1}+r)\}
\]
thus $\mathcal{P}_{n-1}= D \setminus (\bigcup_{i=1}^d D_i)$.

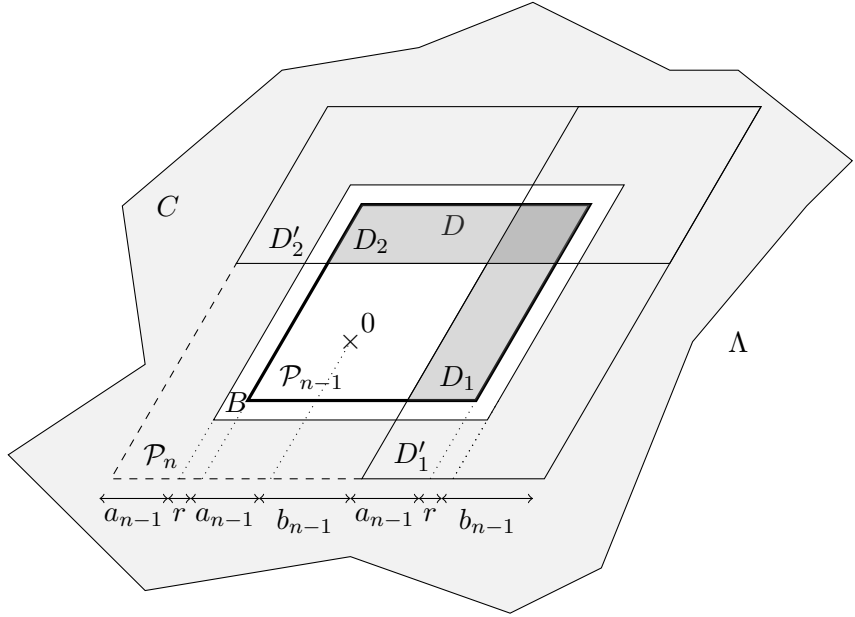
\begin{figure}
 \begin{center}
  \begin{tikzpicture}[scale=0.3]
\draw[black,fill=gray,fill opacity=0.1] (-15,-5)--(-9,-11)--(0,-9.5)--(7,-12)--(11,-10)--(15,0)--(20,6)--(22,8)--(17,12)--(14,12)--(8,15)--(3,13)--(-3,12)--(-10,6)--(-9,-1)--cycle;
\draw (17,0) node{$\Lambda$};
\draw (-8,6) node{$C$};
\draw[black,fill=white] (-6,-3.46)--(6,-3.46)--++(6,10.39)--++(-12,0)--cycle;
\draw (-5,-2.66) node{$B$};
\draw (0,0) node{$\times$} node[above right]{$0$};
\draw(6,3.46)--(2.5,-2.6)--(-4.5,-2.6)--(-1,3.46)--cycle;
\draw (-1.7,-1.6) node{$\mathcal{P}_{n-1}$};
\draw[very thick] (10.5,6.06)--(5.5,-2.6)--++(-10,0)--++(5,8.66)--cycle;
\draw (5.5,6.06) node[below left]{$D$};
\fill[color=gray,opacity=0.3] (-1,3.46)--++(10,0)--++(1.5,2.6)--++(-10,0)--cycle;
\fill[color=gray,opacity=0.3] (2.5,-2.6)--++(3,0)--++(5,8.66)--++(-3,0)--cycle;
\draw (4.7,-1.6) node{$D_1$};
\draw (0.9,4.46) node{$D_2$};
\draw (0.5,-6.06)--++(8,0)--++(9.5,16.45)--++(-8,0)--cycle;
\draw (2.7,-5.06) node{$D_1'$};
\draw (-5,3.46)--++(19,0)--++(4,6.93)--++(-19,0)--cycle;
\draw (-2.8,4.46) node{$D_2'$};
\draw[dashed] (0.5,-6.06)--(-10.5,-6.06)--(-5,3.46);
\draw (-8.3,-5.06) node{$\mathcal{P}_n$};
\draw[<->] (-11,-6.93)--(-8,-6.93) node[midway,below]{$a_{n-1}$};
\draw[<->] (-8,-6.93)--(-7,-6.93) node[midway,below]{$r$};
\draw[<->] (-7,-6.93)--(-4,-6.93) node[midway,below]{$a_{n-1}$};
\draw[<->] (-4,-6.93)--(0,-6.93) node[midway,below]{$b_{n-1}$};
\draw[<->] (0,-6.93)--(3,-6.93) node[midway,below]{$a_{n-1}$};
\draw[<->] (3,-6.93)--(4,-6.93) node[midway,below]{$r$};
\draw[<->] (4,-6.93)--(8,-6.93) node[midway,below]{$b_{n-1}$};
\draw[dotted] (-7.5,-6.06)--(-6,-3.46);
\draw[dotted] (4.5,-6.06)--(6,-3.46);
\draw[dotted] (-6.5,-6.06)--(-4.5,-2.6);
\draw[dotted] (3.5,-6.06)--(5.5,-2.6);
\draw[dotted] (4.5,-6.06)--(6.5,-2.6);
\draw[dotted] (-3.5,-6.06)--(0,0);
    \end{tikzpicture}
  \caption{The setting of lemma \ref{key_lemma}. $C$ is in light gray, $D_1$ and $D_2$ are in darker gray, $D$ 
  is the region with the thick outline.}
  \label{figure_lemme}
 \end{center}
\end{figure}
  
   As in the one-dimensional case, $B$ will be a buffer zone preventing the zeroes of $C$ and $D$ 
   from interacting. In that case, the main reason 
   for which no zero could appear in $B$ was that a zero remained trapped in $\mathcal{P}_{n-1}$, 
   hence there were at most $n-1$ zeroes elsewhere, and $\mathcal{H}_{n-1}$ limited their 
   possible positions. \\
   Here we cannot keep a zero in $\mathcal{P}_{n-1}$, but we can keep a zero in all the 
   $D \setminus D_i$. Indeed, initially there is at least a zero in 
   $\mathcal{P}_{n-1} \subset D \setminus D_i$, 
   and at any time, a zero of $D \setminus D_i$ with the lowest $i$-th coordinate 
   among the zeroes of $D \setminus D_i$ will need an update rule full of zeroes 
   in order to disappear, hence a zero with a $i$-th coordinate as 
   low as its own because there is no update rule whose sites all have positive 
   $i$-th coordinate (this is the reason for which 
   we work in the basis $\{v_1,\dots,v_d\}$). This zero cannot be in $B$ 
   since $B$ remains full of ones, hence it is in $D \setminus D_i$ and so remains in $D \setminus D_i$ 
   at the next step of the path. \\
   This will have the same practical consequences as the zero trapped in $\mathcal{P}_{n-1}$ had 
   in the one-dimensional case: the presence of 
   a zero in each of the $D \setminus D_i$ prevents $\eta^m$ from being $1_\Lambda$; the $n-1$ zeroes that 
   any of the $D_i$, or $C$, may contain will not escape the $D_i$ or $C$. 
   Moreover, for any $i \in \{1,\dots,d\}$, the argument that 
   in dimension 1 prevented the zeroes of $\mathcal{P}_{n-1}$ 
   from escaping to the left part of $B$ because there were no update rule contained in $\mathds{N}^*$ 
   will here prevent zeroes from escaping $D$ via the face with the lowest $i$-th coordinate to enter $B$, since
   there is no update rule whose sites all have positive $i$-th coordinates. 
   Therefore the buffer zone $B$ will be preserved.
   
   The details of the proof are very similar to those of the proof of lemma \ref{key_lemma_dim1}, therefore we 
   only detail the changes. 
   
  We have to change the induction hypothesis $\mathcal{H}_j'$, which becomes:
 \begin{itemize}
  \item[$(P_1^j)$] For all $i \in \{1,\dots,d\}$, $\eta_{D \setminus D_i}^j$ contains a zero.
  \item[$(P_2^j)$] $\eta_{B}^j = 1_{B}$.
  \item[$(P_3^j)$] $\eta_C^j 1_{\Lambda \setminus C} \in V(n-1,\Lambda)$.
  \item[$(P_4^j)$] For all $i \in \{1,\dots,d\}$, $\eta_{D_i}^j 1_{D_i' \setminus D_i} \in V(n-1,D_i')$.
 \end{itemize}
 
 When proving the induction, the more complicated geometry forces us to refine the proof of the fact that the case 
 $z \in B$ and $z+X \subset D$ is impossible.
 Since $z \in B$, if we denote by $(z_1,\dots,z_d)$ the coordinates of 
 $z$, there would exist $i \in \{1,\dots,d\}$ such that $z_i < -a_n+a_{n-1}+r$ ($z$ is ``at the left of $B$ 
 for the $i$-th coordinate'') or  $z_i > b_n -(b_{n-1}+r)$ ($z$ is 
 ``at the right of $B$ for the $i$-th coordinate'').
 \begin{itemize}
  \item If $z_i < -a_n+a_{n-1}+r$, we notice that $z+X \subset D$ would imply that 
 $X \subset \{(x_1,\dots,x_d) \in \mathds{R}^d\,|\, x_i > 0 \}$, which is impossible because there is 
 no update rule whose sites all have a positive $i$-th coordinate. 
  \item If $z_i > b_n -(b_{n-1}+r)$, we can use the same argument as in dimension 1 with $D_i$ replacing $D_1$, 
  which yields a contradiction.
 \end{itemize}
We deduce a contradiction in both cases, therefore $z+X \subset D$ is indeed impossible.
  
 Finally, the proof of $(P_1^{j+1})$ when $z \in D$ also deserves a refinement. 
We set $i \in \{1,\dots,d\}$, let us prove that $\eta_{D \setminus D_i}^{j+1}$ contains a zero. 
If $z \in D_i$, then $\eta_{D \setminus D_i}^{j+1} = \eta_{D \setminus D_i}^j$, hence by $(P_1^j)$ 
 $\eta_{D \setminus D_i}^{j+1}$ contains a zero. If $z \in D \setminus D_i$, 
 we use the fact that $X$ cannot be contained in $\{(x_1,\dots,x_d) \in \mathds{R}^d \,|\, x_i > 0\}$, hence 
 there exists a site $z' \in z+X$ such that the $i$-th coordinate of $z'$ is less 
 than or equal to the $i$-th coordinate of $z$. Moreover, we observe that 
 $z+X \subset D \cup B$, and by $(P_2^j)$ $\eta_{B}^j = 1_{B}$, thus $z+X \subset D$, so $z' \in D$. Since
 the $i$-th coordinate of $z'$ is less than or equal to the $i$-th coordinate of $z$ and 
 $z \in D \setminus D_i$, $z' \in D \setminus D_i$. Furthermore, we have $\eta_{z'}^{j+1} = \eta_{z'}^j =0$. 
 Consequently, $\eta_{D \setminus D_i}^{j+1}$ contains a zero.
Therefore $\eta_{D \setminus D_i}^{j+1}$ contains a zero for all 
$i \in \{1,\dots,d\}$, hence $(P_1^{j+1})$ is true.
\end{proof}

\section{Sketch of the proof of proposition \ref{prop_unrooted}}\label{section_unrooted}
For $d=1$, if $\mathcal{U}$ is a supercritical unrooted family, it has no stable direction, therefore 
there must be an update rule contained in $\mathds{N}^*$ and another contained in $-\mathds{N}^*$.
Consequently, as illustrated by figure \ref{figure_intervalle}, if we have an interval $I \subset \mathds{Z}$ 
of zeroes that is sufficiently large, the site $s$ at the right of $I$ can be put at zero with a legal move. 
Then the site $s'$ at the left of the interval can be put at one by a legal move, and $I$ has moved 
to the right by one unit. By having $I$ starting from outside the domain (where there are only zeroes) 
and moving towards the origin in that way, one can put the origin at zero using a bounded number of zeroes, 
whatever the size of the domain.

 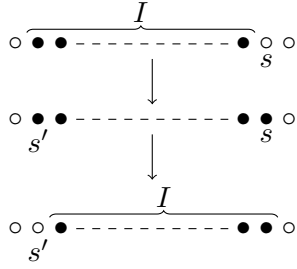
\begin{figure}
 \begin{center}
   \begin{tikzpicture}
   \draw[dashed] (-1,0)--(1,0);
\draw (-1.2,0) node{$\bullet$};
\draw (-1.5,0) node{$\bullet$};
\draw (-1.8,0) node{$\circ$};
\draw (1.2,0) node{$\bullet$};
\draw (1.5,0) node{$\circ$} node[below] {$s$};
\draw (1.8,0) node{$\circ$};
\draw[->] (0,-0.2)--(0,-0.8);
   \draw[dashed] (-1,-1)--(1,-1);
\draw (-1.2,-1) node{$\bullet$};
\draw (-1.5,-1) node{$\bullet$} node[below] {$s'$};
\draw (-1.8,-1) node{$\circ$};
\draw (1.2,-1) node{$\bullet$};
\draw (1.5,-1) node{$\bullet$} node[below] {$s$};
\draw (1.8,-1) node{$\circ$};
\draw[->] (0,-1.2)--(0,-1.8);
   \draw[dashed] (-1,-2.45)--(1,-2.45);
\draw (-1.2,-2.45) node{$\bullet$};
\draw (-1.5,-2.45) node{$\circ$} node[below] {$s'$};
\draw (-1.8,-2.45) node{$\circ$};
\draw (1.2,-2.45) node{$\bullet$};
\draw (1.5,-2.45) node{$\bullet$};
\draw (1.8,-2.45) node{$\circ$};
\draw [decorate,decoration={brace}] (-1.65,0.15)--(1.35,0.15) node[midway,above]{$I$};
\draw [decorate,decoration={brace}] (-1.35,-2.3)--(1.65,-2.3) node[midway,above]{$I$};
  \end{tikzpicture}
  \caption{A move towards the right of an interval $I$ of zeroes 
  for a one-dimensional supercritical unrooted update family. 
  Zeroes are represented by $\bullet$ and ones by $\circ$.}
  \label{figure_intervalle}
 \end{center}
\end{figure}

For $d=2$ the mechanism is similar, but requires a more complex construction. In section 5 of 
\cite{Bollobas_et_al2015} (see in particular figure 5 and lemma 5.5 therein), 
it is proven that if $\mathcal{U}$ is an update family with a semicircle of 
unstable directions centered on direction $u$, it is possible to construct a ``droplet'': a finite set of zeroes that 
even if all other sites are at 1, allows us to put more sites at zero in 
direction $u$ with legal moves, creating a bigger droplet of the same shape, as 
illustrated on part (a) of figure \ref{figure_bootstrap2d}. 
It is the shape of the part of the droplet towards direction $u$ 
that enables its growth towards this direction. 
If $\mathcal{U}$ is supercritical unrooted, its stable directions 
are contained in a hyperplane of $\mathds{R}^2$, which means a straight line, 
hence there are at most two stable directions, and they must then be opposite. 
Therefore, there exists two opposite semicircles containing no stable direction, with middles $u$ and $-u$. 
We can use the construction of \cite{Bollobas_et_al2015} to build two droplets, corresponding to the 
two semicircles, that can grow respectively in the directions $u$ and $-u$ 
(see part (b) of figure \ref{figure_bootstrap2d}). Using these two droplets, we can get a combined droplet 
 that can grow in both directions $u$ and $-u$ (part (c) of figure \ref{figure_bootstrap2d}). Moreover, since 
 our rules allow any change of site state to be reversed, 
the droplet will also be able to shrink in these directions. Therefore, by having the droplet 
grow in direction $u$ and shrink in direction $-u$, we can make it 
move towards direction $u$ (see part (d) of figure \ref{figure_bootstrap2d}). 
This allows us to bring it to the origin using a bounded number of zeroes as we did with the interval we had 
for $d=1$.

\begin{figure}
 \parbox{0.3\textwidth}{
  \begin{center}
 \begin{tikzpicture}[scale=0.8]
  \draw (0,0)--(0,5)--(2,5)--(4,4)--(3,1)--(2,0)--cycle ;
  \draw [dashed] (2,5)--(3,5)--(5,4)--(4,1)--(3,0)--(2,0) ;
  \draw [->] (1,2.5)--(2.5,2.5) node [midway,above] {$u$};
 \end{tikzpicture}
 
 (a)
 \end{center}
 }
 \parbox{0.3\textwidth}{
  \begin{center}
 \begin{tikzpicture}[scale=0.4]
  \draw (0,0)--(0,5)--(2,5)--(4,4)--(3,1)--(2,0)--cycle ;
  \draw [->] (1,2.5)--(2.5,2.5) node [midway,above] {$u$};
 \end{tikzpicture}
 \begin{tikzpicture}[scale=0.4]
  \draw (0,0)--(0,5)--(-2,5)--(-4,4)--(-5,3)--(-5,2)--(-3,0)--cycle;
  \draw [->] (-1.5,2.5)--(-3,2.5) node [midway,above] {$-u$};
 \end{tikzpicture}
 
 (b)
 
 \medskip
 
 \begin{tikzpicture}[scale=0.3]
  \draw (2,5)--(4,4)--(3,1)--(2,0)--(-3,0)--(-5,2)--(-5,3)--(-4,4)--(-2,5)--cycle ;
  \draw [dashed] (2,5)--(4,5)--(6,4)--(5,1)--(4,0)--(2,0) ;
   \draw [dashed] (-2,5)--(-4,5)--(-6,4)--(-7,3)--(-7,2)--(-5,0)--(-3,0) ;
   \draw [->] (2,2.5)--(3,2.5) node [midway,above] {$u$};
   \draw [->] (-3,2.5)--(-4,2.5) node [midway,above] {$-u$};
 \end{tikzpicture}
 
 (c)
 \end{center}
 }
 \parbox{0.3\textwidth}{
  \begin{center}
 \begin{tikzpicture}[scale=0.2]
  \draw (2,5)--(4,4)--(3,1)--(2,0)--(-3,0)--(-5,2)--(-5,3)--(-4,4)--(-2,5)--cycle ;
  \draw (0,-1.5) node{$\downarrow$} ;
  \draw [dashed] (2,-3)--(4,-4)--(3,-7)--(2,-8) ; 
  \draw (2,-8)--(-3,-8)--(-5,-6)--(-5,-5)--(-4,-4)--(-2,-3)--(2,-3) ;
  \draw (2,-3)--(4,-3)--(6,-4)--(5,-7)--(4,-8)--(2,-8) ;
  \draw (0,-9.5) node{$\downarrow$} ;
  \draw (2,-11)--(4,-11)--(6,-12)--(5,-15)--(4,-16)--(-1,-16)--(-3,-14)--(-3,-13)--(-2,-12)--(-0,-11)--cycle ;
  \draw [dashed] (-1,-16)--(-3,-16)--(-5,-14)--(-5,-13)--(-4,-12)--(-2,-11)--(0,-11) ;
 \end{tikzpicture}

 (d)
 \end{center}
 }
 \caption{The construction of a droplet of zeroes for a two-dimensional supercritical unrooted update family 
 that can move towards $u$ and $-u$. (a) The shape delimited by the solid line is the droplet of \protect\cite{Bollobas_et_al2015}, 
 that can grow to the shape delimited by the dashed line. (b) The droplets corresponding to the 
 semicircles centered at $u$ and $-u$. (c) The combined droplet. (d) A move of the combined droplet to the right.}
 \label{figure_bootstrap2d}
\end{figure}
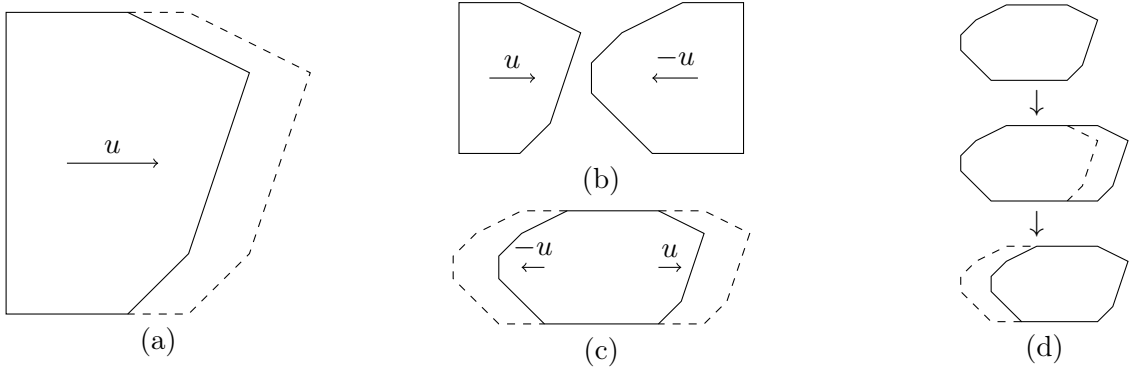

For $d \geq 3$, we expect a similar phenomenon to occur, but we cannot prove it because an equivalent 
of the construction of \cite{Bollobas_et_al2015} is not available yet.

\section*{Appendix: proof of lemma \ref{lemme_base}}

By assumption, the update family $\mathcal{U}$ is not supercritical unrooted, hence its 
stable directions are not contained in any hyperplane of $\mathds{R}^d$. Therefore, 
there exists stable directions $u_1,\dots,u_d$ of $\mathcal{U}$ that form a basis of $\mathds{R}^d$.
For any $u \in S^{d-1}$, we denote $\mathcal{H}_u$ the hyperplane orthogonal to $u$:
$\mathcal{H}_u = \{x \in \mathds{R}^d \,|\, \langle x,u \rangle = 0\}$.
Then, for any $i \in \{1,\dots,d\}$, $\bigcap_{j \neq i} \mathcal{H}_{u_j}$ is a straight line.\footnote{Indeed, 
$\bigcap_{j \neq i} \mathcal{H}_{u_j}$ is the intersection of $d-1$ hyperplanes in $\mathds{R}^d$, 
hence it contains a straight line. Furthermore, $\bigcap_{j \neq i} \mathcal{H}_{u_j}$ is orthogonal 
to the $u_j$, $j \neq i$, and since $\{u_1,\dots,u_d\}$ is a basis of $\mathds{R}^d$, 
$\{u_j : j \neq i\}$ generate a vector space of dimension $d-1$. Therefore $\bigcap_{j \neq i} \mathcal{H}_{u_j}$ 
is orthogonal to a vector space of dimension $d-1$. Consequently, it is at most a straight line.}
For any $i \in \{1,\dots,d\}$, we define $v_i$ as a unitary vector 
 in $\bigcap_{j \neq i} \mathcal{H}_{u_j}$. 

We are going to show that $\{v_1,\dots,v_d\}$ is a basis of $\mathds{R}^d$. 
For any set of vectors $\{w_1,\dots,w_m\} \subset \mathds{R}^d$, we denote $\mathrm{Vect}\{w_1,\dots,w_m\}$ the 
vector space generated by $\{w_1,\dots,w_m\}$. It is enough to prove that 
$\mathrm{Vect}\{v_1,\dots,v_d\}=\mathds{R}^d$. In order to do that, we take $v \in \mathds{R}^d$ 
a vector orthogonal to $\mathrm{Vect}\{v_1,\dots,v_d\}$. We are going to show that 
$v$ must be the null vector. For all $i \in \{1,\dots,d\}$, 
$v$ is orthogonal to $v_i$. Moreover, the vector space orthogonal to $v_i$ has dimension $d-1$. 
Furthermore, $v_i \in \bigcap_{j \neq i} \mathcal{H}_{u_j}$, hence the $u_j$, $j \neq i$ are orthogonal to 
$v_i$. Hence, as the $u_j$, $j \neq i$ are $d-1$ linearly independent vectors, the 
vector space orthogonal to $v_i$ is $\mathrm{Vect}\{u_1,\dots,u_{i-1},u_{i+1},\dots,u_d\}$. 
This implies that $v$ belongs to $\mathrm{Vect}\{u_1,\dots,u_{i-1},u_{i+1},\dots,u_d\}$, for any $i \in \{1,\dots,d\}$.
As $\{u_1,\dots,u_d\}$ is a basis of $\mathds{R}^d$, this yields $v=0$. Consequently, the 
vector space orthogonal to $\mathrm{Vect}\{v_1,\dots,v_d\}$ is reduced to $\{0\}$. We deduce 
$\mathrm{Vect}\{v_1,\dots,v_d\} = \mathds{R}^d$, thus $\{v_1,\dots,v_d\}$ is a basis of $\mathds{R}^d$.

We want a basis such that for any $i \in \{1,\dots,d\}$, $\mathds{H}_{u_i} = 
\{(x_1,\dots,x_d) \in \mathds{R}^d \,|\, x_i > 0\}$. In $\{v_1,\dots,v_d\}$, 
$\{(x_1,\dots,x_d) \in \mathds{R}^d \,|\, x_i = 0\}$ is generated by the vectors $v_1,\dots,v_{i-1},v_{i+1},\dots,v_d$, 
which are $d-1$ linearly independent vectors 
belonging to the hyperplane $\mathcal{H}_{u_i}$ of $\mathds{R}^d$, hence they generate $\mathcal{H}_{u_i}$. 
This implies $\mathcal{H}_{u_i} = 
\{(x_1,\dots,x_d) \in \mathds{R}^d \,|\, x_i = 0\}$. Therefore, $\mathds{H}_{u_i}$ is 
either $\{(x_1,\dots,x_d) \in \mathds{R}^d \,|\, x_i > 0\}$ or 
$\{(x_1,\dots,x_d) \in \mathds{R}^d \,|\, x_i < 0\}$. If $\mathds{H}_{u_i} = 
\{(x_1,\dots,x_d) \in \mathds{R}^d \,|\, x_i < 0\}$, we replace $v_i$ with $-v_i$. Thus we get $\mathds{H}_{u_i} = 
\{(x_1,\dots,x_d) \in \mathds{R}^d \,|\, x_i > 0\}$.

This method allows us to obtain a basis $\{v_1,\dots,v_d\}$ satisfying 
that for any $i \in \{1,\dots,d\}$, $\|v_i\|_2 = 1$ and $\mathds{H}_{u_i} = 
\{(x_1,\dots,x_d) \in \mathds{R}^d \,|\, x_i > 0\}$.

\section*{Acknowledgments}

I would like to thank my PhD advisor Cristina Toninelli for introducing me to the subject and 
helping me along this work. I also would like to thank Robert Morris 
for suggesting that the result could extend beyond 
dimension 2, as well as Fabio Martinelli and Ivailo Hartarsky for their helpful suggestions about the 
presentation of this work.

\end{document}